\documentclass{amsart}
\usepackage{amscd}
\usepackage{amssymb}
\input xypic
\xyoption{all}
\newdir{ >}{{}*!/-5pt/@{>}}

\renewcommand{\top}{\rm top}

\newcommand{\con}{\mathrm{con}}

\def\cA{\mathcal A}

\def\cC{\mathcal C}
\def\cB{\mathcal B}

\def\ob{\mathrm{ob}}

\def\obc{\ob\cC}
\def\cD{\mathcal D}

\def\cE{\mathcal E}

\def\cF{\mathcal F}

\def\fall{\mathcal All}
\def\cG{\mathcal G}

\def\cL{\mathcal L}

\def\cS{\mathcal S}

\def\ass{\mathrm{Alg}}

\def\cat{\mathrm{Cat}}
\def\inj{\mathrm{inj}}
\def\spt{\mathrm{Spt}}

\def\hotimes{\hat{\otimes}}

\newcommand{\C}{\mathbb{C}}

\newcommand{\Q}{\mathbb{Q}}

\newcommand{\fin}{\mathcal{F}in}
\newcommand{\vcyc}{\mathcal{V}cyc}

\newcommand{\Z}{\mathbb{Z}}

\newcommand{\hofi}{\mathrm{hofiber}}

\newcommand{\org}{\mathrm{Or}G}

\def\lra{\longrightarrow}
\def\onto{\twoheadrightarrow}

\def\triqui{\vartriangleleft}
\def\weq{\overset\sim\lra}
\def\lweq{\overset\sim\longleftarrow}
\def\fibeq{\overset\sim\onto}

\numberwithin{equation}{section}
\theoremstyle{plain}
\newtheorem{thm}[equation]{Theorem}
\newtheorem{cor}[equation]{Corollary}

\newtheorem{prop}[equation]{Proposition}

\newcommand{\comment}[1]{}  %to comment out chunks of text

\theoremstyle{definition}

\theoremstyle{remark}
\newtheorem{rem}[equation]{Remark}

\newtheorem{exs}[equation]{Examples}
\newtheorem{nota}[equation]{Notation}

\newtheorem{stan}[equation]{Standing Assumptions}

\begin{document}

\title{Operator ideals and assembly maps in $K$-theory}
\author{Guillermo Corti\~nas}
\email{gcorti@dm.uba.ar}\urladdr{http://mate.dm.uba.ar/\~{}gcorti}
\author{Gisela Tartaglia}
\email{gtartag@dm.uba.ar}
\address{Dep. Matem\'atica-IMAS, FCEyN-UBA\\ Ciudad Universitaria Pab 1\\
1428 Buenos Aires\\ Argentina}
\thanks{The first author was partially supported by MTM2007-64704. Both authors
were supported by CONICET, and partially supported by
grants UBACyT 20020100100386, PIP 112-200801-00900 and MathAmSud
project U11Math05.}
\begin{abstract}
Let $\cB$ be the ring of bounded operators in a complex, separable Hilbert space. For $p>0$ consider the Schatten ideal
$\cL^p$ consisting of those operators whose sequence of singular values is $p$-summable; put $\cS=\bigcup_p\cL^p$. Let $G$ be a group and $\vcyc$ the family of virtually cyclic subgroups. Guoliang Yu proved that the $K$-theory assembly map
\[
H_*^G(\cE(G,\vcyc),K(\cS))\to K_*(\cS[G])
\]
is rationally injective. His proof involves the construction of a
certain Chern character tailored to work with coefficients $\cS$ and
the use of some results about algebraic $K$-theory of operator
ideals and about controlled topology and coarse geometry. In this
paper we give a different proof of Yu's result. Our proof uses the
usual Chern character to cyclic homology. Like Yu's, it relies on
results on algebraic $K$-theory of operator ideals, but no
controlled topology or coarse geometry techniques are used. We
formulate the result in terms of homotopy $K$-theory. We prove that
the rational assembly map
\[
H_*^G(\cE(G,\fin),KH(\cL^p))\otimes\Q\to KH_*(\cL^p[G])\otimes\Q
\]
is injective. We show that the latter map is equivalent to the
assembly map considered by Yu, and thus obtain his result as a
corollary.
\end{abstract}

\maketitle
\section{Introduction}
Let $G$ be a group; a \emph{family} of subgroups of $G$ is a
nonempty family $\cF$ closed under conjugation and under taking
subgroups. If $\cF$ is a family of subgroups of $G$,  then a
$G$-simplicial set $X$ is called a {\em $(G,\cF)$-complex} if the
stabilizer of every simplex of $X$ is in $\cF$. The category of
$G$-simplicial sets can be equipped with a closed model structure
where an equivariant map $X\to Y$ is a weak equivalence (resp. a
fibration) if $X^H\to Y^H$ is a weak equivalence (resp. a fibration)
for every $H\in\cF$ (see \cite[\S1]{corel}); the $(G,\cF)$-complexes
are the cofibrant objects in this model structure. By a general
construction of Davis and L\"uck (see \cite{dl}) any functor $E$
from the category $\Z$-$\cat$ of small $\Z$-linear categories to the
category $\spt$ of spectra which sends category equivalences to
equivalences of spectra gives rise to an equivariant homology theory
of $G$-spaces $X\mapsto H^G(X,E(R))$ for each unital ring $R$, such
that if $H\subset G$ is a subgroup, then
\begin{equation}\label{intro:cross}
H_*^G(G/H,E(R))=E_*(R[H])
\end{equation}
is just $E_*$ evaluated at the group algebra. The \emph{isomorphism
conjecture} for the quadruple $(G,\cF,E,R)$ asserts that if
$\cE(G,\cF)\fibeq pt$ is a $(G,\cF)$-cofibrant replacement of the
point, then the induced map
\begin{equation}\label{intro:assem}
H_*^G(\cE(G,\cF),E(R))\to E_*(R[G])
\end{equation}
--called \emph{assembly map}-- is an isomorphism. For the family
$\cF=\fall$ of all subgroups, \eqref{intro:assem} is always an
isomorphism. The appropriate choice of $\cF$ varies with $E$.  For
$E=K$, the nonconnective algebraic $K$-theory spectrum, one takes
$\cF=\vcyc$, the family of virtually cyclic subgroups. If $E=KH$ is
homotopy $K$-theory, one can equivalently take $\cF$ to be either
$\vcyc$ or the family $\fin$ of finite subgroups (\cite[Thm.
2.4]{blr}). If $E$ satisfies certain hypothesis, including excision,
one can make sense of the map \eqref{intro:assem} when $R$ is
replaced by any, not necessarily unital ring $A$. These hypothesis
are satisfied, for example when $E=KH$. Under milder hypothesis,
which are satisfied for example by $E=K$, \eqref{intro:assem} makes
sense for those coefficient rings $A$ which are \emph{$E$-excisive},
i.e. those for which $E$ satisfies excision (see \cite{corel} and
Subsection \ref{subsec:equigen} below).

The main result of this paper concerns the $KH$-assembly map for
$R=\cL^p$, the Schatten ideal. Recall that $\cL^p$ is an ideal of
the ring $\cB$ of bounded operators in a complex separable Hilbert
space; it consists of those operators whose sequence of singular
values is $p$-summable. Let $\cS=\bigcup_{p>0}\cL^p$; by \cite[Thm.
8.2.1]{CT} (see also \cite[Thm. 4]{wodalg}), $\cS$ is $K$-excisive.
Thus the assembly map  \eqref{intro:assem} with coefficients
$K(\cS)$ makes sense; Guoliang Yu proved in \cite{yu} that it is
rationally injective. His proof involves the construction of a
certain Chern character tailored to work with coefficients $\cS$ and
the use of some results about algebraic $K$-theory of operator
ideals (\cite{CT}, \cite{wodalg}), and about controlled topology and
coarse geometry from \cite{bfjr} and \cite{roe}.

In this paper we give a different proof of Yu's result. Our proof
uses the usual Chern character to cyclic homology. Like Yu's, it
relies on results about algebraic $K$-theory of operator ideals from
\cite{CT} and \cite{wodalg}, but no controlled topology or coarse
geometry techniques are used. We formulate the result in terms of
$KH$; we prove:

\begin{thm}\label{intro:thmkh}
Let $p>0$ and $G$ a group. Then the rational
assembly map
\[
H_*^G(\cE(G,\fin),KH(\cL^p))\otimes\Q\to KH_*(\cL^p[G])\otimes\Q
\]
is injective.
\end{thm}

Yu's result follows as a corollary.

\begin{cor}\label{intro:cor}(\cite[Thm. 1.1]{yu}).
Let $G$ be any group and let $\cS=\bigcup_{p>0}\cL^p$ be the ring of
all Schatten operators. Then the rational assembly map
\[
H_*^G(\cE(G,\vcyc),K(\cS))\otimes\Q\to K_*(\cS[G])\otimes\Q
\]
is injective.
\end{cor}

The proof of the corollary makes it clear that the two assembly maps
are isomorphic, so it really is the same result.

\goodbreak

\smallskip

The rest of this paper is organized as follows. In Section
\ref{sec:equihom}, after recalling some general facts about
equivariant homology spectra, we focus on the cases of periodic
cyclic homology and Hochschild homology. For example, we show in
Proposition \ref{prop:hpcompu} that if $X$ is a $(G,\fin)$-complex
and $k\supset\Q$ a field, then
\[
H^G_n(X,HP(k/k))=\bigoplus_{p\in\Z}H_{n+2p}^G(X,HH(k/k) )
\]
We use this to show in Proposition \ref{prop:asshp} that the
assembly map
\begin{equation}\label{intro:asshp}
H^G_n(\cE(G,\fin),HP(k/k))\to HP_n(k[G]/k)
\end{equation}
is injective for every group $G$. In Section \ref{sec:chern} we
consider the Connes-Karoubi Chern character
\[
ch:H^G(X,KH(A))\to H^G(X,HP(A/k))
\]
defined in \cite[\S8]{corel}. We show in Proposition \ref{prop:yug}
that the composite of $ch$ with the operator trace gives an
equivalence
\[
c:H^G(X,KH(\cL^1))\otimes\C\weq H^G(X,HP(\C/\C))
\]
for every $(G,\fin)$-complex $X$. From this and the fact that
$KH_*(\cL^1)\cong KH_*(\cL^p)$ we deduce --in Corollary
\ref{cor:yug}-- that a similar equivalence which we also call $c$
holds for every $p>0$ :
\begin{equation}\label{intro:ciso}
c:H^G(X,KH(\cL^p))\otimes\C\weq H^G(X,HP(\C/\C))
\end{equation}
Section \ref{sec:kh} is concerned with Theorem \ref{intro:thmkh},
which we prove in Theorem \ref{thm:novikh}. The proof uses
\eqref{intro:ciso} and the injectivity of \eqref{intro:asshp}. Yu's
result \ref{intro:cor} is proved in Corollary \ref{cor:root}.

\begin{nota}
By a spectrum we understand a sequence $E=\{{}_nE:n\ge 1\}$ of
simplicial sets and bonding maps $\Sigma({}_nE)\to {}_{n+1}E$; thus
our notation differs from that of other authors (e.g. \cite{tho})
who use the term prespectrum for such an object. If $E,F:C\to \spt$
are functorial spectra, then by a (natural) \emph{map} $f:E\weq F$
we mean a zig-zag of natural maps
\[
E=Z_0\overset{f_1}\longrightarrow Z_1\overset{f_2}\longleftarrow
Z_2\overset{f_3}\longrightarrow\dots  Z_n=F
\]
such that each right to left arrow $f_i$ is an object-wise weak
equivalence. If also the left to right arrows are object-wise weak
equivalences, then we say that $f$ is a \emph{weak equivalence} or
simply an \emph{equivalence}. If $E$ and $F$ are spectra, we write
$E\oplus F$ for their wedge or coproduct. The Dold-Kan
correspondence associates a spectrum to every chain complex of
abelian groups. Although our notation does not distinguish a chain
complex from the spectrum associated to it, it will be clear from
the context which of the two we are referring to.
\goodbreak

Rings in this paper are not assumed unital, unless explicitly
stated. We use the letters $A,B$ for rings, and $R,S$ for unital
rings. If $X$ is a set, then $M_X$ is the ring of all matrices
$(z_{x,y})_{x,y\in X\times X}$ with integer coefficients, only
finitely many of which are nonzero. If $A$ is a ring, then
$M_XA=M_X\otimes A$; in particular $M_X\Z=M_X$. If $A$ and $B$ are
rings, then $A\oplus B$ is their direct sum as abelian groups,
equipped with coordinate-wise multiplication.

\goodbreak
\end{nota}

\section{Equivariant cyclic homology}\label{sec:equihom}
\numberwithin{equation}{subsection}
\subsection{Equivariant homology of simplicial
sets}\label{subsec:equigen} Let $k$ be a field. A \emph{$k$-linear
category} is a small category enriched over the category of
$k$-vector spaces. We write $k$-$\cat$ for the category of
$k$-linear categories and $k$-linear functors. Observe that, by
regarding a unital $k$-algebra as $k$-linear category with one
object, we obtain a fully faithful embedding of $k$-algebras into
$k$-$\cat$. Let $\cC\in k$-$\cat$, consider the $k$-module
\[
\cA(\cC)=\bigoplus_{x,y\in \cC}\hom_\cC(x,y)
\]
If $f\in \cA(\cC)$ write $f_{a,b}$ for the component in
$\hom_\cC(b,a)$. The following multiplication law
\begin{equation}\label{rule:matrix}
(fg)_{a,b}=\sum_{c\in\obc}f_{a,c}g_{c,b}
\end{equation}
makes $\cA(\cC)$ into an associative $k$-algebra, which is unital if
and only if $\obc$ is finite. Whatever the cardinal of $\obc$ is,
$\cA(\cC)$ is always a ring with {\it local units}, i.e. a filtering
colimit of unital rings. We call $\cA(\cC)$ the \emph{arrow ring} of
$\cC$. If  $F:\cC\mapsto \cD$ is a $k$-linear functor which is
injective on objects, then it defines a homomorphism
$\cA(F):\cA(\cC)\to \cA(\cD)$ by the rule $\alpha\mapsto F(\alpha)$.
Hence we may regard $\cA$ as  a functor
\begin{equation}\label{fun:ac}
\cA:\inj\emph{-k-}\cat\to \emph{k-}\ass
\end{equation}
from the category of $k$-linear categories and functors which are
injective on objects, to the category of $k$-algebras. However
$\cA(F)$ is not defined for general $k$-linear $F$. Let
$E:k$-$\cat\to\spt$ be a functor. If $R$ is a unital $k$-algebra and
$I\triqui R$ is a $k$-ideal, we put
\[
E(R:I)=\hofi(E(R)\to E(R/I))
\]
Thus if we assume $E(0)\weq *$, we have $E(R:R)\weq E(R)$. We say
that $E$ is \emph{finitely additive} if the canonical map

\[
E(\cC)\oplus E(\cD)\to E(\cC\oplus\cD)
\]

is an equivalence. We assume from now on that $E$ is finitely
additive. If $A$ is a not necessarily unital $k$-algebra, write
\[
\tilde{A}_k=A\oplus k
\]
for the unitization of $A$ as a $k$-algebra. Put
\[
E(A)=E(\tilde{A}_k:A)
\]
If $A$ happens to be unital, we have two definitions for $E(A)$;
they are equivalent by \cite[Lemma 1.1]{kabi}. A not necessarily
unital ring $A$ is called \emph{$E$-excisive} if for any embedding
$A\triqui R$ as an ideal of a unital $k$-algebra, the canonical map
\[
E(A)\to E(R:A)
\]
is an equivalence.

\begin{stan}\label{stan}
We shall henceforth assume that $E:k$-$\cat\to\spt$ satisfies each of
the following.
\goodbreak

\smallskip

\begin{itemize}

\item[i)] Every algebra with local units is $E$-excisive.

\item[ii)] If $H$ is a group and $A$ an $E$-excisive algebra, then $A[H]$ is
$E$-excisive.

\item[iii)] If $A$ is $E$-excisive, $X$ a set and $x\in X$, then $M_{X}A$ is $E$-excisive, and $E$ sends the map
$A\to M_{X}A$, $a\mapsto e_{x,x}a$ to a weak equivalence.

\item[iv)] There is a natural weak equivalence $E(\cA(\cC))\weq E(\cC)$ of functors \\ $\inj$-$k$-$\cat\to\spt$.

\item[v)] Let $A$ and $B$ be algebras, and let $C=A\oplus B$
be their direct sum, with coordinate-wise multiplication. Then $C$
is $E$-excisive if and only if  both $A$ and $B$ are. Moreover if
these equivalent conditions are satisfied, then the map $E(A)\oplus
E(B)\to E(C)$ is an equivalence.
\end{itemize}
\end{stan}

\begin{exs}\label{exs:hck}
The assumptions above can be formulated for linear categories over
any commutative, unital ground ring $k$. The (nonconnective)
$K$-theory spectrum $K$ satisfies the standing assumptions for
$k=\Z$ as well as for any field $k$ of characteristic zero
(\cite[Prop. 4.3.1, Prop. 6.4]{corel}). The homotopy $K$-theory
spectrum $KH$ is \emph{excisive}, i.e. every ring is $KH$-excisive
\cite{kh}. Furthermore it satisfies the assumptions over any ground
ring $k$ (\cite[Prop. 5.5]{corel}). A $k$-linear category $\cC$ has
associated a canonical cyclic $k$-module $C(\cC/k)$ (\cite{machc})
with
\[
C(\cC/k)_n=\bigoplus_{(c_0,\dots,c_n)\in\obc^{n+1}}\hom_\cC(c_1,c_0)\otimes_k\dots\otimes_k\hom_\cC(c_0,c_n)
\]
The Hochschild, cyclic, negative cyclic and periodic cyclic homology
of $\cC$ over $k$ are the respective homologies of $C(\cC/k)$; they
are denoted $HH(/k)$, $HC(/k)$, $HN(/k)$ and $HP(/k)$. Both $HH(/k)$
and $HC(/k)$ satisfy the assumptions above when $k$ is any field
\cite[Prop. 6.4]{corel}. If $k$ is a field of characteristic zero,
then $HP(/k)$ is excisive \cite{cq}; furthermore, it satisfies the
standing assumptions because $HC(/k)$ does. It follows that also
$HN(/k)$ satisfies the assumptions.
\end{exs}

Let $G$ be a group, and $S$ a $G$-set. Write $\cG^G(S)$ for its
\emph{transport groupoid}. By definition $\ob\cG^G(S)=S$, and
$\hom_{\cG^G(S)}(s,t)=\{g\in G:g\cdot s=t\}$. If $R$ is a unital
$k$-algebra we consider a small category $R[\cG^G(S)]$. The objects
of $R[\cG^G(S)]$ are those of $\cG^G(S)$ and
\[
hom_{R[\cG^G(S)]}(s,t) = R\otimes\Z[hom_{\cG^G(S)}(s,t)]
\]
with the obvious composition rule. Note that $R[\cG^G(S)]$ is a
$k$-linear category. We write $\org$ for the orbit category of $G$;
its objects are the $G$-sets $G/H$, $H\subset G$ a subgroup; its
homomorphisms are the $G$-equivariant maps. The rule $G/H\mapsto
R[\cG^G(G/H)]$ defines a functor $\org\to k$-$\cat$.
\begin{rem}\label{arnuni}
If $A$ is any, not necessarily unital $k$-algebra we put
\[
\cA(A[\cG^G(G/H)])= \ker (\cA(\tilde{A}_k[\cG^G(G/H)]) \to \cA(k[\cG^G(G/H)]))
\]
Note that $\cA(A[\cG^G(G/H)])$ is always defined, even though $A[\cG^G(G/H)]$ is not. Morover, by \cite[Lemma 3.2.6]{corel} there is an isomorphism $$\cA(A[\cG^G(G/H)]) \weq M_{G/H}(A[H]).$$
\end{rem}

\bigskip

\noindent
If $R$ is a unital algebra and $I\triqui R$ is a $k$-ideal, put
\[
E(R[\cG^G(G/H)]:I[\cG^G(G/H)])=\hofi(E(R[\cG^G(G/H)])\to
E((R/I)[\cG^G(G/H)]))
\]
In the particular case when $A$ is an $E$-excisive algebra, we set
\[
E(A[\cG^G(G/H)]) = E(\tilde{A}_k[\cG^G(G/H)]:A[\cG^G(G/H)])
\]
If $X$ is a $G$-simplicial set, we consider the coend of spectra
\begin{equation}\label{dlequihomo}
H^G(X,E(R:I))=\int^{\org}X^H_+\land E(R[\cG^G(G/H)]:I[\cG^G(G/H)])
\end{equation}
To abbreviate notation, we write $H^G(X,E(R))$ for $H^G(X,E(R:R))$.
The spectrum $H^G(X,E(R))$ is a simplicial set version of the
Davis-L\"uck equivariant homology spectrum associated with $E$
(\cite{dl},\cite{corel}). We have a fibration sequence
\[
H^G(X,E(R:I))\to H^G(X,E(R))\to H^G(X,E(R/I))
\]
If $A$ is $E$-excisive, put
\begin{equation}\label{equihomoexci}
H^G(X,E(A)):=H^G(X,E(\tilde{A}_k:A))
\end{equation}
By \cite[Prop. 3.3.9]{corel}, if
\[
0\to A'\to A\to A^{\prime\prime}\to 0
\]
is an extension of $E$-excisive algebras, then
\begin{equation}\label{Efibseq}
E(A'[\cG^G(-)])\to E(A[\cG^G(-)])\to E(A^{\prime\prime}[\cG^G(-)])
\end{equation}
and
\begin{equation}\label{Efibseq2}
H^G(X,E(A'))\to H^G(X,E(A))\to H^G(X,E(A^{\prime\prime}))
\end{equation}
are homotopy fibrations.

\subsection{Equivariant periodic cyclic homology of
$(G,\fin)$-complexes}\label{subsec:equihp}
\begin{prop}\label{prop:hpcompu}
Let $X$ be a $(G,\fin)$-complex, and let $k\supset \Q$ be a field.
There is a natural quasi-isomorphism
\[
\bigoplus_{p\in\Z}H^G(X,HH(k/k))[2p]\weq H^G(X,HP(k/k))
\]
\end{prop}
\begin{proof}
It suffices to show that there is a quasi-isomorphism
\[
\bigoplus_{p\in\Z}H^G(G/H,HH(k/k))[2p]\weq H^G(G/H,HP(k/k))
\]
for $H\in\fin$, natural with respect to $G$-equivariant maps. In
particular we may restrict to proving the proposition for $X$ a
discrete, $G$-finite $(G,\fin)$-complex. The cyclic module
$C(k[\cG^G(X)]/k)$ decomposes into a direct sum of cyclic modules
\cite[7.1]{corel}
\[
C(k[\cG^G(X)]/k)=\bigoplus_{[g]\in\con G}C^{[g]}(k[\cG^G(X)]/k)
\]
Here the direct sum runs over the set $\con G$ of conjugacy classes
of elements of $G$. Because $X$ is a $(G,\fin)$-complex,
$C^{[g]}(k[\cG^G(X)]/k)=0$ for $g$ of infinite order. So assume $g$
is of finite order. Let $Z_g\subset G$ be the centralizer subgroup.
By \cite[Lemma 7.2]{corel} (see also \cite[Cor. 9.12]{LR1}), there is a natural quasi-isomorphism of
cyclic $k$-modules
\[
H(Z_g,k[X^g])\to C^{[g]}(k[\cG^G(X)]/k)
\]
Here the domain has the cyclic structure given by
\begin{gather*}
t_n:H(Z_g,k[X^g])_n=k[Z_g]^{\otimes n}\otimes_k k[X^g]\to
H(Z_g,k[X^g])_n\\
t_n(z_1\otimes\dots\otimes z_n\otimes x)=(z_1\cdots z_n)^{-1}g\otimes
z_1\otimes\dots\otimes z_{n-1}\otimes z_n(x)
\end{gather*}
Let $\langle g\rangle\subset Z_g$ be the cyclic subgroup. A Serre
spectral sequence argument (see the proof of \cite[Lemma
9.7.6]{chubu}) shows that the projection
\[
H(Z_g,k[X^g])\to H(Z_g/\langle g\rangle,k[X^g])
\]
is a quasi-isomorphism of cyclic $k$-modules. Summing up we have a
natural zig-zag of quasi-isomorphisms of cyclic modules
\begin{equation}\label{map:alfiso}
\bigoplus_{[g]\in\con_f G}H(Z_g/\langle
g\rangle,k[X^g])\lweq\bigoplus_{[g]\in\con_f
G}H(Z_g,k[X^g])\weq C(k[\cG^G(X)]/k)
\end{equation}
Here $\con_fG$ is the set of conjugacy classes of elements of finite order.
We remark that, because $X$ is a finite $(G,\fin)$-complex by
assumption, the direct sums above have only finitely many nonzero
summands. By \cite[Corollary 9.7.2]{chubu}, we have a natural
equivalence
\[
HP(H(Z_g/\langle g\rangle,k[X^g]))\weq \prod_{p\in\Z}H(Z_g/\langle
g\rangle,k[X^g])[2p]
\]
Summing up we have a natural quasi-isomorphism
\begin{equation}\label{map:hpgx}
\prod_{p \in \Z}\left(\bigoplus_{[g]\in\con_f G}
H(Z_g,k[X^g])\right)[2p]\weq HP(k[\cG^G(X)/k])\weq H^G(X,HP(k/k))
\end{equation}
Taking into account \eqref{map:alfiso} we obtain a quasi-isomorphism
of chain complexes
\begin{equation}\label{map:equihomox}
\bigoplus_{[g]\in\con_f G} H(Z_g,k[X^g])\weq H^G(X,HH(k/k))
\end{equation}
Moreover, in \eqref{map:hpgx} we can replace $\prod_{p\in\Z}$ by
$\bigoplus_{p \in \Z}$
 because $H_n^G(X,HH(k/k))=0$ for $n\ne 0$. Indeed, $X$ is a finite
disjoint union of homogeneous spaces $G/K$ with $K \in \fin$, and
\begin{equation*}
H_n^G(G/K,HH(k/k))=H_n^K(K/K,HH(k/k))=HH_n(k[K]/k)
\end{equation*}
which is zero in positive dimensions since $k[K]$ is separable for
finite $K$. This concludes the proof.
\end{proof}

\begin{prop}\label{prop:asshp}(cf. \cite[Rmk. 1.9]{LR1})
If $k\supset\Q$ is a field, then the assembly map
\[
H_*^G(\cE(G,\fin),HP(k/k))\to HP_*(k[G]/k)
\]
is injective.
\end{prop}

\begin{proof}
The inclusion
\[
C(k[G]/k)=\bigoplus_{[g]\in\con G}C^{[g]}(k[G]/k)\subset
\prod_{[g]\in\con G}C^{[g]}(k[G]/k)
\]
induces a chain map $HP(k[G]/k)\to \prod_{[g]\in\con
G}HP^{[g]}(k[G]/k)$. Projecting onto the conjugacy classes of
elements of finite order and taking homology we obtain a
homomorphism
\begin{equation}\label{map:projecthp}
HP_n(k[G]/k)\to \prod_{[g]\in\con_fG}HP^{[g]}_n(k[G]/k)
\end{equation}
Now for $g$ of finite order $C^{[g]}(k[G]/k)=H(Z_g,k)\weq
H(Z_g/\langle g\rangle,k)$, hence by \cite[Cor. 9.7.2]{chubu}
\[
HP^{[g]}_n(k[G]/k)=\prod_{m}H_{n+2m}(Z_g,k[pt])=\prod_{m}H_{n+2m}(Z_g,k[\cE(G,\fin)^g])
\]
One checks that the composite of the assembly map with the map
\eqref{map:projecthp} is the inclusion
\begin{gather*}
H_n^G(\cE(G,\fin),HP(k/k))=
\bigoplus_m
H_{n+2m}^G(\cE(G,\fin),HH(k/k))\\
\subset\prod_m
H_{n+2m}^G(\cE(G,\fin),HH(k/k))=\prod_m\bigoplus_{[g]\in\con_f G}
H_{n+2m}(Z_g,k[\cE(G,\fin)^g])\\
\subset \prod_m \prod_{[g]\in\con_fG}H_{n+2m}(Z_g,k[\cE(G,\fin)^g])
\end{gather*}
The first identity above is the case of $X=\cE(G,\fin)$ of Proposition \ref{prop:hpcompu} and the second follows from \eqref{map:equihomox}.
\end{proof}

\section{Equivariant Connes-Karoubi Chern character}\label{sec:chern}
\numberwithin{equation}{section}

In this section we consider algebras over a field $k$ of
characteristic zero. Recall from \cite[\S8.2]{corel} that the
homotopy $K$-theory and periodic cyclic homology of a $k$-linear
category are related by a Connes-Karoubi Chern character
\begin{equation}\label{chern}
\xymatrix{ KH(\cC)\ar[r]^{ch}& HP(\cC/k)}
\end{equation}
In particular if $G$ is a group, $H\subset G$ a subgroup and $R$ a
unital $k$-algebra we have a map of $\org$-spectra
\[
ch:KH(R[\cG^G(G/H)])\to HP(R[\cG^G(G/H)]/k)
\]
By \cite[Lemma 3.2.6]{corel}, this map is equivalent to the Chern
character
\[
ch:KH(R[H])\to HP(R[H]/k)
\]
for each fixed $H$. Using excision, all this extends to an arbitrary
nonunital algebra $A$ in place of $R$.
 We are interested in the
particular case when $k$ is either $\C$ or $\Q$ and $A\triqui \cB$
is an ideal in the algebra $\cB$ of bounded operators in a separable
complex Hilbert space. Let $p>0$; write $\cL^p\triqui\cB$ for the
Schatten ideal of those compact operators whose sequence of singular
values is $p$-summable. Let $H\subset G$ be a subgroup and $Tr:\cL^1
\to \C$ the operator trace. The map of cyclic modules
\begin{gather*}
Tr:C(\widetilde{\cL^1}_\C[\cG^G(G/H)]:\cL^1[\cG^G(G/H)]/\C)\to C(\C[\cG^G(G/H)]/\C)\\
Tr(a_0\otimes g_0\otimes\dots\otimes a_n\otimes g_n)=Tr(a_0\cdots
a_n)g_0\otimes\dots\otimes g_n
\end{gather*}
induces a natural transformation of $\org$-chain complexes
\begin{equation}\label{map:trace}
Tr:HP(\cL^1[\cG^G(G/H)]/\C)\to HP(\C[\cG^G(G/H)]/\C)
\end{equation}

\begin{prop}\label{prop:yug}
Let $X$ be a $(G,\fin)$-complex and $\cL^1\triqui\cB$ the ideal of
trace class operators. Then the composite
\[
\xymatrix{c:H^G(X,KH(\cL^1))\otimes\C\ar[r]^(.55){ch}&
H^G(X,HP(\cL^1/\C))\ar[r]^{Tr}& H^G(X,HP(\C/\C))}
\]
is an equivalence.
\end{prop}
\begin{proof}
It suffices to consider the case $X=G/H$ with $H\in\fin$. By
\cite[Lemma 3.2.6]{corel}, we have a homotopy commutative diagram
with vertical equivalences
\[
\xymatrix{KH(\cL^1[H])\otimes\C\ar[d]^\wr\ar[r]&HP(\cL^1[H]/\C)\ar[r]\ar[d]^\wr&HP(\C[H]/\C)\ar[d]^\wr\\
KH(\cL^1[\cG^G(G/H)])\otimes\C\ar[r]&HP(\cL^1[\cG^G(G/H)]/\C)\ar[r]
&HP(\C[\cG^G(G/H)]/\C)}
\]
Because $H\in\fin$, $\C[H]$ is Morita equivalent to its center,
which is a sum of copies of $\C$ indexed by the conjugacy classes of
$H$:
\begin{equation}\label{maschke}
\C[H]\sim Z(\C[H])=\bigoplus_{\con(H)}\C
\end{equation}
Since the (periodic) cyclic homology of $\C$ as a $\C$-algebra and
as a locally convex topological algebra agree, it follows that the
map $HP(\C[H]/\C)\to HP^{\top}(\C[H])$ is the identity. Next recall from
\cite[Notation 5.1 and Theorem 6.2.1 (iii)]{CT} and that if $I\triqui \cB$ is an operator ideal with a Banach algebra structure such that the inclusion $I\subset \cB$ is continuous and $\hotimes$ is the projective tensor product, then for every locally convex algebra $A$
the comparison map $KH(I\hotimes A)\to K^{\top}(I\hotimes A)$ is an equivalence. In particular this applies when $I=\cL^1$
and $A=\C[H]$. Hence we have a homotopy commutative diagram with
vertical equivalences
\[
\xymatrix{
KH(\cL^1[H])\otimes\C\ar[d]^\wr\ar[rr]^{c}&&HP(\C[H]/\C)\ar[d]^\wr\\
K^{\top}(\cL^1[H])\otimes\C\ar[r]_{ch^{\top}}&HP^{\top}(\cL^1[H])\ar[r]_{Tr}&HP^{\top}(\C[H])}
\]
Here $ch^{\top}$ is the topological Connes-Karoubi Chern character.
Using \eqref{maschke} we have that $Tr$ is an equivalence by
\cite[Prop. 17.2]{cumf} and $ch^{\top}$ is an equivalence because of
\eqref{maschke} and of the commutativity of the following diagram

\[
\xymatrix{K^{\top}(\cL^1)\otimes\C\ar[r]^{ch^{\top}}&HP^{\top}(\cL^1/\C)\\
K^{\top}(\C)\otimes\C\ar[u]^{\wr}\ar[r]^{ch^{\top}}_{\sim}&HP^{\top}(\C/\C)\ar[u]^{\wr}}
\]

It follows that $c$ is an equivalence. This concludes the proof.
\end{proof}

\begin{cor}\label{cor:yug}
Let $X$ be a $(G, \fin)$ complex. Then, for every $p>0$ there is an equivalence
\[
c:H^G(X, KH(\cL^p)) \otimes \C \rightarrow H^G(X,HP(\C/\C))
\]
\end{cor}
\begin{proof}Because $\cL^1/\cL^p$ for $p<1$ and $\cL^p/\cL^1$ for
$p>1$ are nilpotent rings and $KH$ is nilinvariant (\cite{kh}), the maps $$KH(\cL^p[\cG^G(-)])\to
KH(\cL^1[\cG^G(-)])\quad (p<1)$$ and $$KH(\cL^1[\cG^G(-)])\to
KH(\cL^p[\cG^G(-)]) \quad(p>1)$$ are equivalences of $\org$-spectra by Remark  \ref{arnuni} and \eqref{Efibseq}. The
proof is now immediate from Proposition \ref{prop:yug}.
\end{proof}

\section{The $KH$-assembly map with
$\cL^p$-coefficients}\label{sec:kh}

\begin{thm}\label{thm:novikh}
Let $p>0$ and $G$ a group. Then the
rational assembly map
\[
H_*^G(\cE(G,\fin),KH(\cL^p))\otimes\Q\to KH_*(\cL^p[G])\otimes\Q
\]
is injective.
\end{thm}
\begin{proof}
It suffices to show that the map tensored with $\C$ is injective. We
have a commutative diagram
\[
\xymatrix{ H_*^G(\cE(G,\fin),KH(\cL^p))\otimes\C\ar[d]_{c}^\wr\ar[r]&
KH_*(\cL^p[G])\otimes\C\ar[d]^{c}\\
H_*^G(\cE(G,\fin),HP(\C/\C))\ar@{ >->}[r]& HP_*(\C[G]/\C)}
\]
The vertical map on the left is an isomorphism by Corollary
\ref{cor:yug}; the bottom horizontal map is injective by Proposition
\ref{prop:asshp}. It follows that the top horizontal map is
injective. This concludes the proof.
\end{proof}

Let $\cS=\bigcup_{p>0}\cL^p$ be the ring of all Schatten operators. Because $\cS^2=\cS$, the ring $\cS$ is $K$-excisive by \cite[Thm. C]{sw} and \cite[Proof of Thm. 8.2.1]{CT} (see also \cite[Thm. 4]{wodalg}). We can now deduce the following result of Guoliang Yu.

\begin{cor}\label{cor:root}(\cite[Thm. 1.1]{yu}).
The rational assembly map
\[
H_*^G(\cE(G,\vcyc),K(\cS))\otimes\Q\to K_*(\cS[G])\otimes\Q
\]
is injective.
\end{cor}
\begin{proof} Because $\cS=\cS^2$ and the tensor product of operators $\cB\otimes\cB\to \cB(H\otimes H)\cong \cB$ sends $\cL^1\otimes \cS\to \cS$, the operator ideal $\cS$ is sub-harmonic in the sense of \cite[Def. 6.5.1]{CT}. Hence the map $K(A\otimes_\C \cS)\to KH(A\otimes_\C \cS)$ is an equivalence for every H-unital $\C$-algebra $A$ (\cite[Thm. 8.2.5, Rmk. 8.2.6]{CT}). Applying this when $A=\C[H]$ and using the fact that both $K$ and $KH$ satisfy \ref{stan}, we obtain an equivalence of $\org$-spectra $K(\cS[\cG^G(G/H)])\to
KH(\cS[\cG^G(G/H)])$ (see Remark \ref{arnuni}). Hence for every $G$-simplicial set $X$ we have a homotopy commutative diagram with vertical equivalences
\begin{equation}\label{findiag}
\xymatrix{H^G(X,K(\cS))\ar[d]^\wr\ar[r]^-{assem}&K(\cS[G])\ar[d]^\wr\\
H^G(X,KH(\cS))\ar[r]^-{assem}&KH(\cS[G])}
\end{equation}
On the other hand by \cite[Thm. 7.4]{blr} and \eqref{Efibseq2}, for every ring $A$ the map
$$\cE(G,\fin)\to\cE(G,\vcyc)$$ induces a weak equivalence
\[
H^G(\cE(G,\fin),KH(A))\weq H^G(\cE(G,\vcyc),KH(A))
\]
Applying this to $A=\cS$ and using Theorem \ref{thm:novikh} we obtain that when $X=\cE(G,\vcyc)$ the bottom horizontal arrow in \eqref{findiag} --and thus also the top arrow--  is a rational equivalence.
\end{proof}

\bibliographystyle{plain}

\begin{thebibliography}{10}

\bibitem{bfjr}
A. ~Bartels, T. Farrell, L. Jones, H. Reich.
\newblock On the isomorphism conjecture in algebraic K-theory.
\newblock{\em Topology } 43(1):157--213,  2004.

\bibitem{blr}
A.~Bartels, W.~L\"uck.
\newblock Isomorphism conjecture for homotopy K-theory and groups acting on trees.
\newblock {\em J. Pure Appl. Algebra}, 205(3):660--696, 2006.


\bibitem{kabi}
G.~Corti\~nas.
\newblock The obstruction to excision in {$K$}-theory and in cyclic homology.
\newblock{\em Invent. Math.} 164:143--173, 2006.

\bibitem{corel}
G. ~Corti\~nas, E.~Ellis.
\newblock{Isomorphism conjectures with proper coefficients}.
\newblock{\verb|arXiv:1108.5196|}

\bibitem{CT}
G.~Corti\~nas, A.~Thom.
\newblock {Comparison between algebraic and topological K-theory of locally convex algebras}.
\newblock {\em Adv. Math.}, 218:266--307, 2008.

\bibitem{cumf}
J.~Cuntz.
\newblock{Cyclic theory and the bivariant Chern-Connes
character.} \newblock{\em Noncommutative geometry},  73--135.
\newblock{ Lecture Notes in Math. 1831, Springer, Berlin, 2004.}

\bibitem{cq}
J.~Cuntz, D.~Quillen.
\newblock Excision in bivariant periodic cyclic cohomology.
\newblock \emph{Invent. Math.} 127 (1997), 67–-98.

\bibitem{dl}
J.~Davis and W.~L{\"u}ck.
\newblock {Spaces over a category and assembly maps in isomorphism conjectures
  in $K$- and $L$-theory}.
\newblock {\em $K$-theory} 15:241--291, 1998.


\bibitem{LR1}
W.~L{\"u}ck, H.~Reich. \newblock{ Detecting $K$-theory by cyclic
homology.} \newblock{\em Proc. London Math. Soc.}   93(3):593--634,
2006.

\bibitem{machc} R. McCarthy.
\newblock{The cyclic homology of an exact category.}
\newblock{\em J. Pure Appl. Algebra}
\newblock{ 93 (1994), no. 3, 251--296.}

\bibitem{roe}
J. Roe.
\newblock Coarse cohomology and index theory on complete Riemannian
 manifolds.
\emph{Mem. Amer. Math. Soc.  104}, 1993.

\bibitem{sw}
A. Suslin, M. Wodzicki.
\newblock Excision in algebraic $K$-theory.
\emph{Ann. of Math. 136, 51--122}, 1993.

\bibitem{tho}
R.W. Thomason.
\newblock Algebraic $K$-theory and \'etale cohomology.
\newblock{{\em Annales scientifiques de l' \'E.N.S. $4^{\circ}$ s\'erie},
tome 18, $n^{\circ}3$, 1985, 437--552.}

\bibitem{chubu}
C. Weibel.
\newblock {\em An introduction to homological algebra},
\newblock Cambridge Univ.\ Press, 1994.

\bibitem{kh}
C. Weibel. {\it Homotopy Algebraic $K$-theory}. Contemporary Math.
{\bf 83} (1989) 461--488.

\bibitem{wodalg}
M. Wodzicki.
\newblock Algebraic $K$-Theory and functional analysis,
\newblock{in: First European Congress of Mathematics, Vol. II, Paris, 1992},
\newblock {\em Progr. Math.}, Birkh\"auser, Basel, 120:485--496, 1994.

\bibitem{yu}
G. Yu.
\newblock The algebraic K-theory Novikov conjecture for group algebras
over the ring of Schatten class operators. \verb|arXiv:1106.3796|

\end{thebibliography}

\end{document}